\documentclass{amsart}

\usepackage{amsmath}%
\usepackage{amsfonts}%
\usepackage{amssymb}%
\usepackage{amsthm}
\usepackage{graphicx}
\usepackage{dsfont}
\usepackage{pdfcomment}
\usepackage{tikz}
\usepackage{tikz-3dplot}
\usepackage{subfigure}
\usepackage{soul}
\usepackage{enumerate}
\usepackage{cancel}
\usepackage{stmaryrd}
\usepackage{mathtools}

\usepackage{hyperref}
\usepackage{cleveref}
\usepackage{esvect}

\usetikzlibrary{decorations.markings}
\newtheorem{theorem}{Theorem}[section]

\newtheorem{corollary}[theorem]{Corollary}

\newtheorem{definition}[theorem]{Definition}

\allowdisplaybreaks

\newtheorem{proposition}[theorem]{Proposition}

\DeclareMathOperator{\Ad}{Ad}

\DeclareMathOperator{\tr}{tr}

\DeclareMathOperator{\Real}{Re}
\DeclareMathOperator{\Imaginary}{Im}

\DeclareMathOperator{\Hom}{Hom}
\DeclareMathOperator{\Teich}{Teich}
\renewcommand{\Re}{\Real}
\renewcommand{\Im}{\Imaginary}

\title{Space of circle patterns on tori and its symplectic form}
\author{Wai Yeung Lam}
\thanks{This work was partially supported by the FNR grant CoSH O20/14766753.}
\address{Department of Mathematics, University of Luxembourg, Maison du nombre, 6 avenue de la Fonte, L-4364 Esch-sur-Alzette, Luxembourg.} \email{wyeunglam@gmail.com}

\begin{document}

		\maketitle

	\begin{abstract}
We consider circle patterns on closed tori equipped with complex projective structures. There is an embedding of the space of circle patterns to the Teichm\"{u}ller space of a punctured surface. Via the embedding, the Weil-Petersson symplectic form is pulled back to the space of circle patterns. We investigate its non-degeneracy. On the other hand, we also complete a conjecture that the space of circle patterns is homeomorphic to the Teichm\"{u}ller space of the closed torus.

	\end{abstract}

\section{Introduction}

Conformal maps in the plane are characterized as mappings sending infinitesimal circles to themselves. Instead of infinitesimal size, a circle packing in the plane is a configuration of finite-size circles where certain pairs are mutually tangent. Thurston proposed regarding the map induced from two circle packings with the same tangency pattern as a discrete conformal map \cite{Stephenson2005}. With circle packings, one could define discrete conformal structures and compare them with classical conformal structures.

Generalizing the notion of circle packings, a circle pattern in the plane is a realization of a planar graph such that each face has a circumcircle passing through the vertices. By adding diagonals, we assume each face is triangular and the circle pattern is determined by cross ratios as follows. Assume $(V, E, F)$ is a triangulation of the planar graph. Given a realization of the vertices $z:V \to \mathbb{C}\cup\{\infty\}$, we associate a complex cross ratio to the common edge $\{ij\}$ shared by triangles $\{ijk\}$ and $\{jil\}$ 
\[
X_{ij} :=  -\frac{(z_k - z_i)(z_l -z_j)}{(z_i - z_l)(z_j - z_k)} =X_{ji} \in \mathbb{C}
\]
It defines a function  $X: E  \to \mathbb{C}$ satisfying certain polynomial equations, which can be defined generally on any surface. 
\begin{definition}\label{def:crsys}
	Suppose $(V,E,F)$ is a triangulation of a closed surface $S_g$ of genus $g$, where $V$, $E$ and $F$ denote respectively the sets of vertices, edges and faces. An unoriented edge is denoted by $\{ij\}=\{ji\}$ indicating that its end points are vertices $i$ and $ j$. A cross ratio system is an assignment $X:E \to \mathbb{C}$ such that for every vertex $i$ with adjacent vertices numbered as $1$, $2$, ..., $r$ in the clockwise order counted from the link of $i$,
	\begin{gather} 
		\Pi_{j=1}^n X_{ij} =1  \label{eq:crproduct}\\
		X_{i1} + X_{i1} X_{i2} + X_{i1}X_{i2}X_{i3} + \dots +  X_{i1}X_{i2}\dots X_{ir} =0 \label{eq:crsum}
	\end{gather}
	where $X_{ij} = X_{ji}$. 
\end{definition}

A cross ratio system provides a recipe to lay out neighboring circumdisks. Equations \eqref{eq:crproduct} and \eqref{eq:crsum} ensure that the holonomy around each vertex under the gluing construction is the identity. 

We are interested in circle patterns with prescribed Delaunay intersection angles. Observe that the imaginary part $\Im (\log X) : E \to [0,2\pi)$ is the intersection angles of circumdisks. 

\begin{definition} Given a triangulation of $S_g$, a Delaunay angle structure is an assignment of angles $\Theta:E \to [0,\pi)$ with $\Theta_{ij}=\Theta_{ji}$ satisfying the following:   
	\begin{enumerate}[(i)]
		\item For every vertex $i$,  \[ \sum_j \Theta_{ij} = 2\pi\] where the sum is taken over the neighboring vertices of $i$ on the universal cover.
		\item For any collection of edges $(e_0,e_1,e_2,\dots,e_r=e_0)$ whose dual edges form a simple closed contractable path on the surface, then
		\[
		\sum_{i=1}^{n} \Theta_{e_i} > 2\pi
		\]
		unless the path encloses exactly one primal vertex.
	\end{enumerate}
	We consider cross ratio systems with prescribed Delaunay angles 
	\[
	P(\Theta) = \{ X: E \to \mathbb{C}| \Im (\log X) = \Theta \text{ and satisfying } \eqref{eq:crproduct}\eqref{eq:crsum} \}.
	\]
\end{definition}    
It is known that these conditions on $\Theta$ are necessary and sufficient for $P(\Theta)$ to be non-empty \cite{Rivin}. Generally, every Delaunay cross ratio system in $P(\Theta)$ induces a complex projective structure on the closed surface. We denote $P(S_g)$ the space of marked complex projective structure on the closed surface. By forgetting the circle pattern and considering the underlying complex projective structure, one obtains a mapping from the space of circle patterns to the Teichm\"{u}ller space of the closed surface via a composition of maps
\[
P(\Theta) \xrightarrow{f} P(S_g) \xrightarrow{\pi} \Teich(S_g)
\]
where $f$ is the forgetful map and $\pi$ is the classical uniformization map for Riemann surfaces. It is conjectured that $\pi \circ f$ is a homeomorphism \cite{KMT2003}. For $g>1$, the conjecture is largely open. There are partial results but it is not yet clear whether $P(\Theta)$ is a manifold in general \cite{BW2023,Lam2024s,Schlenker2018}. For the torus case $g=1$, it is better understood.

\begin{theorem}[\cite{Lam2021}]
 Given any Delaunay angle structure $\Theta$ on a torus $S_1$, the mapping \[ \pi\circ f:P(\Theta)- \{X^{\dagger}\} \to  \Teich(S_1)- \{\tau^{\dagger}\}\] is a finite-sheet covering map where $X^{\dagger}$ is the unique cross ratio system that induces an Euclidean torus and $\tau^{\dagger}$ is the associated conformal structure. In particular, the projection \[ \pi\circ f:P(\Theta)\to  \Teich(S_1)\] is a covering map with at most one branch point. 
\end{theorem}

The proof involves verifying the local diffeomorphism with the use of the discrete Laplacian. However the argument fails at Euclidean structures, where the character variety of $P(S_1)$ becomes singular \cite{BG2005}. We remedy it by taking a topological approach.
\begin{theorem}\label{thm:main1}
	The covering map 
	 \[ \pi\circ f:P(\Theta)- \{X^{\dagger}\} \to  \Teich(S_1)- \{\tau^{\dagger}\}\]
	 has degree $1$ and hence 
	 \[ \pi\circ f:P(\Theta)\to  \Teich(S_1)\] 
	 is a homeomorphism.
\end{theorem}

On the other hand, there is an embedding of $P(\Theta)$ to the Teichm\"{u}ller space of a punctured surface $\Teich(S_{g,n})$ with $n=|V|$, which consists of marked complete hyperbolic metrics with cusps at the punctures. The complex cross ratios $X$ provide a recipe to construct a pleated surface in hyperbolic 3-space by gluing ideal triangles with shear coordinates $\Re \log X$ and bending angles $\Im \log X$. By forgetting the bending angles, every circle pattern induces a complete hyperbolic metric with cusps on the punctured surface. Thus we have an embedding
\[
\tilde{f}:P(\Theta)  \xhookrightarrow{} \Teich(S_{g,n}) \cong \mathbb{R}^{6g-6+2n}
\]
Via the embedding $\tilde{f}$, the Weil-Petersson symplectic form $\omega_{P}$ on $\Teich(S_{g,n}) $ is pulled back to $P(\Theta)$.

In \cite{Lam2024s}, we related it to Goldman's symplectic form on $P(S_g)$. It was conjectured that the pullback is non-degenerate, and hence defines a symplectic structure on the space of circle patterns. With the tool of the discrete Laplacian, we verify it in the case of tori.

\begin{theorem}\label{thm:main2}
	Given any Delaunay angle structure $\Theta$ on a torus, the pullback of the symplectic form $\tilde{f}^*\omega_{P}$ is non-degenerate on $P(\Theta)-\{X^{\dagger}\}$.
\end{theorem}

Although our approach of discrete Laplacian is not applicable at Euclidean tori, we believe that the bilinear form is still non-degenerate there. In the last section, we illustrate it with Euclidean tori obtained from triangular lattices.

\section{Complex projective structures on tori}

We recall the complex projective structures on tori and show that $\pi\circ f$ is a homeomorphism (Theorem \ref{thm:main1}).

\begin{definition}
	A complex projective structure on a closed surface $S_{g}$ is a maximal atlas of charts from open subsets of $S_{g}$ to the Riemann sphere such that the transition functions are restrictions of M\"{o}bius transformations. These charts are called projective charts.
	
	Two complex projective structures are marked isomorphic if there is a diffeomorphism homotopic to the identity mapping projective charts to projective charts. We denote $P(S_{g})$ the space of marked complex projective structures up to isomorphism.
\end{definition}

Given a complex projective structure on a torus $S_1$, there is a developing map of the universal cover to the Riemann sphere and a holonomy representation of the fundamental group $\pi_{1}(S_1)$ 
\[
\rho \in \Hom(\pi_{1}(S_1),PSL(2,\mathbb{C})).
\]
Since the fundamental group of the torus is abelian, the holonomy of the complex projective torus has one common fixed point or two depending on whether the torus admits an Euuclidean structure. By normalizing one of the fixed point at infinity, one obtains a complex affine structure. 

\begin{definition}
	A complex affine structure on a surface $S$ is a maximal atlas of charts from open subsets of $S$ to $\mathbb{C}$ such that the transition functions are restrictions of complex affine transformations $z \mapsto a z + b$ for some $a,b \in \mathbb{C}$ with $a \neq 0$.
\end{definition}

Indeed, all complex projective structures can be reduced to complex affine structures.

\begin{proposition}[Gunning \cite{Gunning}]\label{prop:gunning}
	Every complex projective structure on a torus can be reduced to an affine structure.
\end{proposition}

Throughout the paper, we denote $\gamma_1,\gamma_2 \in \pi_{1}(S)$ generators of the fundamental group of the torus and write their holonomy representation $\rho_1= \rho(\gamma_1)$ and $\rho_2= \rho(\gamma_2)$ in $SL(2,\mathbb{C})$. By moving a common fixed point to infinity, the holonomy representation of a complex projective can be normalized to one of the followings. 
\begin{enumerate}[(I)]
	\item (Euclidean tori) The holonomy $\rho_1,\rho_2$ are translation, i.e. $\exists \beta_r \in \mathbb{C}$ such that \[
	(z\circ \gamma_r)_i = \rho_r (z_i) = z_i + \beta_r  \quad \forall \, i \in \hat{V}.\]
	\item (Non-Euclidean tori) The holonomy becomes stretched rotation, i.e. $\exists \alpha_r \in \mathbb{C}$ such that \[
	(z\circ \gamma_r)_i = \rho_r (z_i) = \alpha_r z_i  \quad \forall \, i \in \hat{V}.\]
\end{enumerate}

Circle patterns on affine tori can be parametrized using the affine holonomy.

\begin{proposition}[{\cite[Theorem 7.2]{Rivin}}] \label{prop:rivin}
	Let $\Theta$ be a Delaunay angle structure on a triangulated torus and $A_1,A_2 \in \mathbb{R} $. Then there exists a unique affine structure on the torus with affine holonomy  $\rho_r(z)=\alpha_r z + \beta_r$ such that $\log |\alpha_r|= A_r$ that support a unique circle pattern with cross ratio $X \in P(\Theta)$. The developing map $z$ is unique up to a global affine transformation.
\end{proposition}

Two different affine structures might correspond to the same complex projective structure. The affine holonomy always have a common fixed point at $\infty$. As long as it shares another fixed point in $\mathbb{C}$, we can apply an inversion to interchange the two fixed points and obtain a new affine structure, while the underlying complex projective structure remains the same, and so is the cross ratio. In the notation of Proposition \ref{prop:rivin}, an affine holonomy shares two fixed points if and only if $(A_1,A_2)\neq (0,0)$. 

\begin{corollary}\label{cor:puncplane}
For any Delaunay angle structure $\Theta$ on a triangulated surface, we have a homeomorphism given via Proposition \ref{prop:rivin}
\[
P(\Theta) \cong \Omega:= \mathbb{R}^2/ \sim 
\]
where the quotient is given by an equivalence relation $(A_1,A_2) \sim (-A_1,-A_2)$.
\end{corollary}

Given a Delaunay cross ratio system on the torus, it induces a complex projective structure and thus a conformal structure. The underlying conformal structure can be read off from the affine developing map. Recall that we can parameterize the affine structures with a fixed underlying conformal structure on the torus by a complex parameter $c$ in the upper half plane $\mathbb{H}$ as follows: Start with a Euclidean torus obtained by gluing the opposite sides of a parallelogram spanned by complex numbers $1$ and $\tau$ in the upper half plane where the horizontal side and the vertical side form the loops $\gamma_1$ and $\gamma_2$ after gluing. The parameter $\tau$ represent a marked conformal structure in the Teichmm\"{u}ller space $\Teich(S_1) \cong \mathbb{H}$. Let $d:\tilde{S}\to \mathbb{C}$ be a developing map of an affine structure with the same marked conformal structure. Its holonomy satisfies $d(z+1)= \alpha_1 d(z) + \beta_1$ and $d(z+\tau) = \alpha_2 d(z) + \beta_2$. Notice that $d$ is holomorphic and $d' \neq 0$. We have $d''/d'$ holomorphic and periodic on the torus. Thus $d''/d' = c$ for some constant $c\in \mathbb{C}$. In the case $c=0$, $d(z) =az + b$ and we get a Euclidean torus. In the case $c\neq 0$, we have $d(z) = a e^{c z} + b$ for some constants $a,b$, which can be normalized to $d(z)= e^{cz}$ by translation and scaling. Its developing map satisfies $d(z+1) = e^c d(z)$ and $d(z+\tau)= e^{c\tau} d(z)$. Thus the holonomy is generated by 
\begin{gather*}
	z\circ \gamma_1 = \rho_1(z) = e^{c}z \\ z\circ \gamma_2= \rho_2(z) = e^{c \tau}z
\end{gather*}

By comparing it with Proposition \ref{prop:rivin}, we see that given any real numbers $(A_1,A_2)\neq (0,0)$, there is a unique circle pattern $X \in P(\Theta)$ on an affine torus with conformal structure $\tau$ and affine parameter $c$ such that
\[
\Re c = A_1, \quad \Re (c \tau) = A_2.
\]
On the other hands, it is known that the imaginary parts take values in a bounded interval.
\begin{proposition}[{\cite[Lemma 4.1]{Lam2021}}]\label{lem:bounded}
	For any Delaunay cross ratio system, both $|\Imaginary c|$ and $|\Imaginary c \tau|$ are bounded by a constant depending only on the triangulation.
\end{proposition}

We take a topological approach to show that $\pi \circ f: P(\Theta) \to \Teich(S_1)$ is a homeomorphism. The following argument is due to Tianqi Wu.

\begin{proof}[Proof of Theorem \ref{thm:main1}]
	
   We parameterize $\Teich(S_1)$ by the upper half plane $\mathbb{H}$ and $P(\Theta)$ by $\Omega$ using Corollary \ref{cor:puncplane}. Then it suffices to consider $\pi\circ f: \Omega \to \mathbb{H}$. Let $\tau_0 = \pi \circ f(0,0)$ be the conformal structure associated to the unique Euclidean torus that supports a circle pattern with intersection angle $\Theta$. 
   
   Writing $\Omega_{0}:= \Omega- \{0,0\}$ and $\mathbb{H}_0-\{\tau_0\}$, it is known from \cite{Lam2021} that $\pi \circ f: \Omega_{0} \to \mathbb{H}_0$ is a covering map. We shall further argue that the covering map has degree $1$. For any fixed $R>0$, we consider a simple loop $\gamma_R$ on $\Omega_0$ given by
   \[
   \gamma_R(t)= (R \cos t, R \sin t)
   \] 
   for $t \in [0,\pi]$. It represents a generator in $\pi_{1}(\Omega_0)$. 
   
   We claim that $\pi \circ f \circ \gamma_R$ is a loop in $\mathbb{H}_{0}$ representing a generator of $\pi_{1}\left(\mathbb{H}_{0}\right)$ as well.
Observe that
\[
\pi \circ f\left(A_{1}, A_{2}\right)=\frac{c \tau}{c}=\frac{\operatorname{Re}(c \tau)+i \cdot \operatorname{Im}(c \tau)}{\operatorname{Re}(c)+i \cdot \operatorname{Im}(c)}=\frac{A_{2}+i \cdot \operatorname{Im}(c \tau)}{A_{1}+i \cdot \operatorname{Im}(c)}
\]
where both $c, \tau$ are continuous functions of $\left(A_{1}, A_{2}\right)$ and $|\operatorname{Im}(c)|,|\operatorname{Im}(c \tau)|$ are bounded. To make the picture clearer, we map the upper half plane to the unit disk via a homeomorphism $g(z)=\frac{z-i}{z+i}$. Then for sufficiently large $R$
\[
g \circ \pi \circ f \circ \gamma_R(t)=\frac{A_{2}-i A_{1}+O(1)}{A_{2}+i A_{1}+O(1)}=\frac{R \sin t-i R \cos t+O(1)}{R \sin t+i R \cos t+O(1)}=e^{i \cdot 2 t}+O\left(R^{-1}\right).
\]
When $t \in[0, \pi]$, it forms the desired loop that generates $\pi_{1}(g(\mathbb{H}_0))$. Thus the covering map $\pi \circ f: \Omega_{0} \to \mathbb{H}_0$ has degree $1$ and hence $\pi\circ f:\Omega \to \mathbb{H}$ is a homeomorphism.
\end{proof}

\section{ Weil-Petersson symplectic form on $\Teich(S_{1,n})$}

We recall the theory of decorated Teichm\"{u}ller space of punctured surfaces \cite{Penner2012}, particularly punctured tori. Given a triangulation $(V,E,F)$ of a torus $S_1$, we can interpret it as an ideal triangulation of a punctured torus $S_{1,n}:=S_1-V$ by removing vertices, where $n:=|V|$ is the number of vertices.

We denote $\Teich(S_{1,n})$ the Teichm\"{u}ller space  consisting of marked complete hyperbolic metrics with cusps at the $n$ punctures. Particularly, it is known that
\[
\Teich(S_{1,n}) \cong  \mathbb{R}^{2n}.
\]
With the triangulation fixed, the Teichm\"{u}ller space can be parametrized by \emph{positively real} cross ratios $X:E \to \mathbb{R}_{>0}$ on edges as follows. Given a complete hyperbolic metric with cusps, there is a developing map of the universal cover to the hyperbolic plane in the Poincare disk model. For any edge $ij$, we focus on one of its lifts to the universal cover, with neighbouring triangles $ijk$ and $jil$. Via the developing map, the four corners are mapped to $z_j,z_k,z_i,,z_l \in S^{1} \subset \mathbb{C}$ counterclockwisely. The cross ratio 
\[
X_{ij} :=  -\frac{(z_k - z_i)(z_l -z_j)}{(z_i - z_l)(z_j - z_k)}  \in \mathbb{R}_{>0}
\]
satisfying $X_{ij}=X_{ji}$ is independent of the choice of the lift. Hence it defines a function $X:E \to \mathbb{R}_{>0}$. Furthermore, we have for every vertex $i$ with adjacent vertices numbered as $1$, $2$, ..., $n$ in the clockwise order counted from the link of $i$,
\begin{gather}
	\Pi_{j=1}^n X_{ij} =1  \label{eq:prodx}.
\end{gather}
Geometrically, the quantity $\log X_{ij}$ is called the shear coordinate, which is the signed hyperbolic distance between the tangency points of the geodesic $ij$ with the respective incircles in the ideal triangles $ijk$ and $jkl$. Conversely, given a function  $X:E \to \mathbb{R}_{>0}$ satisfying \eqref{eq:prodx}, one obtains a complete hyperbolic metric with cusps by gluing ideal triangles with shear coordinates $\log X$.

A tangent vector to the Teichm\"{u}ller space can be described by the logarithmic derivative of the cross ratio, i.e. $x:= \frac{d}{dt} (\log X^{(t)})|_{t=0}$ where $X^{(t)}$ represents a path in $\Teich(S_{g,n})$. In such a way, every tangent vector corresponds to an element in a real vector space
\[
W:=\{ x \in \mathbb{R}^{E}| \forall i \in V,   \sum_j x_{ij}=0\}.
\]
We have an identification of the tangent space
\begin{align}\label{eq:Tteich}
	T_X \Teich(S_{g,n}) \cong W.
\end{align}

In order to introduce the symplectic form, we consider the decorated Teichm\"{u}ller space $\widetilde{\Teich}(S_{1,n})$. It consists of decorated hyperbolic structures, each of which represents a point in $\Teich(S_{g,n})$ together with a choice of horocycle $H_i$ for each puncture $i \in V$. By considering the hyperbolic length of the horocycles at the punctures, one has
\[
\widetilde{\Teich}(S_{g,n}) = \Teich(S_{g,n}) \times \mathbb{R}^{n}_{>0}.
\]
Any decorated hyperbolic structure yields a function  $A:E \to \mathbb{R}_{>0}$ such that $\log A_{ij}$ is the signed hyperbolic distance between the horocycles $H_i$ and $H_j$. It takes positive sign whenever the horocycles are disjoint. Such functions parametrize the decorated Teichm\"{u}ller space.

Similarly, a tangent vector to the decorated Teichm\"{u}ller space can be described by the logarithmic derivative, i.e. $a:= \frac{d}{dt} (\log A^{(t)})|_{t=0}$ where $A^{(t)}$ represents a path in $\Teich(S_{g,n})$. We have an identification of the tangent space
\[
T_A \widetilde{\Teich}(S_{g,n}) \cong \mathbb{R}^{E}.
\]

By forgetting the horocycles, there is a natural projection
\begin{align*}
	\widetilde{\Teich}(S_{g,n}) &\to \Teich(S_{g,n}) \\
	A_{ij} &\mapsto X_{ij} = \frac{A_{ki} A_{lj}}{A_{il} A_{jk}} 
\end{align*}
and for the tangent space
\begin{align*}
	T_A \widetilde{\Teich}(S_{g,n}) &\to T_X \Teich(S_{g,n}) \\
	a_{ij} &\mapsto  x_{ij} = a_{ki}-a_{il}+a_{lj}-a_{jk}
\end{align*}
Particularly, the linear map 
\begin{align*}
	h:\mathbb{R}^{E} &\to W \\
	a& \mapsto x_{ij} = a_{ki}-a_{il}+a_{lj}-a_{jk}
\end{align*}
is surjective.

\begin{theorem}[Penner]\label{thm:penner}
	The pullback of the Weil-Petersson symplectic 2-form $\omega_P$ on $\Teich(S_{g,n})$ is a bilinear form on $\widetilde{\Teich}(S_{g,n})$ 
	\begin{align*}
		\tilde{\omega}_{P} &:= -2 \sum_{ijk \in F} d \log A_{ij} \wedge d \log A_{jk} + d \log A_{jk} \wedge d \log A_{ki} + d \log A_{ki} \wedge d \log A_{ij}. 
	\end{align*}
	It is  invariant under change of horocycles and invariant under edge flipping in the triangulation. 
\end{theorem}

In terms of the logarithmic derivative $a, \tilde{a} \in T_A \widetilde{\Teich}(S_{g,n}) $
\[
\tilde{\omega}_{P}(a,\tilde{a}) =
-2 \sum_{ijk \in F} a_{ij} (\tilde{a}_{jk}-\tilde{a}_{ki}) + a_{jk} (\tilde{a}_{ki}-\tilde{a}_{ij}) + a_{ki} (\tilde{a}_{ij}-\tilde{a}_{jk}).
\]

\begin{corollary}\label{cor:penner}
The Weil-Petersson symplectic form on $\Teich(S_{g,n})$ can be written as follows: $\forall x,\tilde{x} \in W \cong T_X \Teich(S_{g,n})$
\[
\omega_{P}(x,\tilde{x}) = \widetilde{\omega}_{P}(a,\tilde{a})
\]
where $a \in h^{-1}(x)$, $\tilde{a} \in h^{-1}(\tilde{x})$.	
\end{corollary}

	For any fixed Delaunay angle structure $\Theta$ on the torus, the mapping from the space of circle patterns with prescribed intersection angles $\Theta$ to the Teichm\"{u}ller space of punctured torus 
	\begin{align*}
		\tilde{f}: P(\Theta) &\to \Teich(S_{g,n}) \\
		X &\mapsto |X|
	\end{align*}
	is an embedding. For every $X \in P(\Theta)$ there is a natural inclusion $T_X P(\Theta) \subset T_{|X|} \Teich(S_{1,n})$. The pullback of $\omega_P$ is a closed two-form on $P(\Theta)$. For it to be a symplectic form on $P(\Theta)$, it remains to check whether it is non-degenerate everywhere.

In \cite{Lam2024s}, the first author expressed the pullback of $\omega_P$ on $P(\Theta)$ in terms of the change of holonomy. We adapt to expression the the case of tori.

\begin{proposition}[{\cite{Lam2024s}[Theorem 1.3]}] Suppose $\Theta$ is a Delaunay angle structure on the torus. Let $X \in P(\Theta)$ and $\rho \in \Hom(\pi_{1},SL(2,\mathbb{C}))$ be a holonomy representation of the complex projective structure. Then for any infinitesimal deformations $x,\tilde{x} \in T_X P(\Theta)$,
\[
 \frac{1}{2}\tilde{f}^*\omega_{P}(x,\tilde{x}) =  \tr \left( \tau_1 \Ad \rho_1 (\tilde{\tau}_2) - \tau_2 \Ad \rho_2 (\tilde{\tau}_1) \right)
\]
where $\tau_r:= \dot{\rho}_r \rho^{-1}_r$ is the logarithmic change of holonomy under the infinitesimal deformation given by $x$.
\end{proposition}

We focus on complex affine tori that are non-Euclidean.
\begin{corollary}\label{cor:pullback} Let $X \in P(\Theta)$ represent a non-Euclidean affine tori with holonomy $\rho$ of the form for $r=1,2$
	\begin{equation}\label{eq:holmat}
			\rho_r = \left(\begin{array}{cc}
			e^{\frac{A_r+ \mathbf{i} B_r}{2}} & 0 \\  0 & e^{-\frac{A_r+ \mathbf{i} B_r}{2}}
		\end{array} \right)
	\end{equation}
	for some $A_r,B_r \in \mathbb{R}$. Given any $x \in T_X P(\Theta)$, we write the change in the affine holonomy as
	\[
	\dot{\rho}_r \rho^{-1}_r =  \left(\begin{array}{cc}
		\frac{\mathfrak{a}_r+ \mathbf{i} \mathfrak{b}_r}{2} & 0 \\  0 & 	-\frac{\mathfrak{a}_r+ \mathbf{i} \mathfrak{b}_r}{2}
	\end{array} \right)
	\]
	where $\mathfrak{a}_r=\dot{A}_r,\mathfrak{b}_r=\dot{B}_r$ for $r=1,2$. Then the pullback of the symplectic form on $T_X P(\Theta)$ is
	\[
	\frac{1}{2}\tilde{f}^*\omega_{P}(x,\tilde{x}) =  (\mathfrak{a}_1+ \mathbf{i} \mathfrak{b}_1) (\tilde{\mathfrak{a}}_2+ \mathbf{i} \tilde{\mathfrak{b}}_2) - (\mathfrak{a}_2+ \mathbf{i} \mathfrak{b}_2) (\tilde{\mathfrak{a}}_1+ \mathbf{i} \tilde{\mathfrak{b}}_1).
	\]
\end{corollary}

Following Proposition \ref{prop:rivin} and Corollary \ref{cor:puncplane}, the tangent plane $T_X P(\Theta) \cong \mathbb{R}^2$ is parametrized by the real part $(\mathfrak{a}_1,\mathfrak{a}_2)=(\dot{A}_1,\dot{A}_2)$ for a non-Euclidean affine torus $X$. As discussed in the next section, the imaginary parts $\mathfrak{b}_1, \mathfrak{b}_2$ depend on $\mathfrak{a}_1,\mathfrak{a}_2$ via the discrete harmonic conjugate.

\section{Action on the period space via harmonic conjugate}\label{sec:harcon}
\subsection{Classical harmonic conjugate}

We recall the classical action of harmonic conjugate on the period space. It hold generally for closed Riemann surfaces $S_g$ but we limit our discussion to tori $S_1$. Given closed loops $\gamma_1, \gamma_2$ that generates the fundamental group. It is known that for any real numbers $\mathfrak{a}_1, \mathfrak{a}_2$, there exists a unique harmonic 1-form $\eta$ such that 
\[
\int_{\gamma_1} \eta = \mathfrak{a}_1 \quad \int_{\gamma_2} \eta = \mathfrak{a}_2.
\]
The space of harmonic 1-forms is isomorphic to the period space $\mathbb{R}^{2}$.

With the conformal structure $\tau$, there is a natural action on the period space via the harmonic conjugate. Suppose $\eta$ is a harmonic 1-form with periods  $\mathfrak{a}_1, \mathfrak{a}_2$. Then its harmonic conjugate $\star\eta$ is again a harmonic 1-form, where $\star$ is the Hodge star operator on differential forms. We denote $\mathfrak{b}_1,\mathfrak{b}_2$ the period of $\star \eta$. We thus obtain a linear map
\begin{align*}
	\mathfrak{h}_{\tau} : \mathbb{R}^{2} &\to \mathbb{R}^{2} \\
	(\mathfrak{a}_1,\mathfrak{a}_2) & \mapsto (\mathfrak{b}_1,\mathfrak{b}_2)
\end{align*}
Recall that up to a multiple constant, there is only one holomorphic 1-form on a torus, whose period is in the form of $(c,c \tau)$ for some $c \in \mathbb{C}$. Since $\eta + \mathbf{i} \star\eta$ is a holomorphic 1-form, we could deduce explicitly that 
\begin{align*}
	\left( \begin{array}{c}
		\mathfrak{b}_1 \\ \mathfrak{b}_2
	\end{array}\right) =
	\mathfrak{h}_{\tau} \left( \begin{array}{c}
		\mathfrak{a}_1 \\ \mathfrak{a}_2
	\end{array}\right) = \left( \begin{array}{cc}
	\frac{\Re \tau}{\Im \tau} & -\frac{1}{\Im \tau} \\ \frac{|\tau|^2}{\Im \tau} & -\frac{\Re \tau}{\Im \tau}
\end{array} \right) \left( \begin{array}{c}
	\mathfrak{a}_1 \\ \mathfrak{a}_2
\end{array}\right)
\end{align*}
It can be verified directly that $\mathfrak{h}_{\tau} \circ \mathfrak{h}_{\tau}= -\mbox{Id}$. 
Over the period space, there is a natural symplectic form $\omega: \mathbb{R}^2 \times \mathbb{R}^2 \to \mathbb{R}$ defined such that
\begin{align*}
	\omega(\mathfrak{a},\tilde{\mathfrak{a}}) =   \mathfrak{a}_1 \tilde{\mathfrak{a}}_2 - \mathfrak{a}_2  \tilde{\mathfrak{a}}_1.
\end{align*}
Observe that for any $\mathfrak{a}=(\mathfrak{a}_1,\mathfrak{a}_2) \in \mathbb{R}^2$
\[
\omega( \mathfrak{a}, \mathfrak{h}_{\tau}(\mathfrak{a})) = \iint_S \eta \wedge \star \eta
\]
is the Dirichlet energy of the harmonic 1-form with period $(\mathfrak{a}_1,\mathfrak{a}_2)$ over the Riemann surface. Generally we have
\begin{align*}
	\omega(\mathfrak{a},\tilde{\mathfrak{a}}) =  \iint_S \eta \wedge \tilde{\eta}
\end{align*}
where $\eta, \tilde{\eta}$ are respectively the unique harmonic 1-form with periods $(\mathfrak{a}_1,\mathfrak{a}_2)$ and $(\tilde{\mathfrak{a}}_1,\tilde{\mathfrak{a}}_2)$. Since the Hodge star operator satisfies 
\[
\star \eta \wedge \star  \tilde{\eta} = \eta \wedge \tilde{\eta}
\]
we deduce that $\mathfrak{h}_{\tau}$ is compatible with $\omega$ in the sense that 
\[
	\omega(\mathfrak{h}_{\tau}(\mathfrak{a}),\mathfrak{h}_{\tau}(\tilde{\mathfrak{a}})) = 	\omega(\mathfrak{a},\tilde{\mathfrak{a}}).
\]
Particularly, each conformal structure $\tau$ defines an inner product on the period space $\mathbb{R}^2$
\[
	\langle \mathfrak{a},\tilde{\mathfrak{a}} \rangle_{\tau} :=  	\omega(\mathfrak{a},\mathfrak{h}_{\tau}(\tilde{\mathfrak{a}})) 
\]
and induces a norm on $\mathbb{R}^2$ via
\[
||\mathfrak{a}||_{\tau}:= \sqrt{\langle \mathfrak{a},\mathfrak{a} \rangle_{\tau} }.
\]

\subsection{Discrete harmonic conjugate}

We investigate the action of discrete harmonic conjugate over the period space. We take the parametrizion of $P(\Theta)$ as in Corollary \ref{cor:puncplane}. Fix $(A_1,A_2) \neq (0,0)$. Consider the circle pattern $X \in P(\Theta)$ on a non-Euclidean affine torus with holonomy in the form of \eqref{eq:holmat} having the given $(A_1,A_2)$. Observe that the lifts of a triangular face differ by scaling and hence all corner angles are invariant under the holonomy. It induces edge weights $c:E \to \mathbb{R}$ with the "cotangent formula" \cite{Pinkall1993}: for every edge $\{ij\}$, we define
\[
c_{ij}= \frac{1}{2} (\cot \angle jki +  \cot \angle ilj)
\]
where $\{jki\}, \{ilj\}$ are two triangles sharing the edge $\{ij\}$. The Delaunay condition implies $c_{ij}=c_{ji} \geq 0$ since $ \angle jki +   \angle ilj = \pi- \Theta_{ij} \in (0,\pi]$. The edge weight $c_{ij}$ is equal to zero if and only if the neighbouring faces sharing the same circumcircle. By removing all edges $ij$ such that $\Theta_{ij}=0$, we obtain a Delaunay cell decomposition $(V',E',F')$ (not necessarily triangulated) with the same vertex set and a positive function $c:E' \to \mathbb{R}_{>0}$.

With the edge weights, we can consider discrete harmonic 1-forms on the dual cell decomposition. A discrete 1-form $\eta$ is a function on oriented edges $\eta: \vec{E}' \to \mathbb{R}$ such that $\eta(ij)=-\eta(ji)$ for every edge $ij$. It is closed if its summation around each vertex is zero, .i.e. for every $i \in V'=V$,
\[
\sum_j \eta(ij) =0
\]
where the summation is over edges in $E'$ adjacent to the vertex $i$. Equivalently, it can be regarded as the summation over all edges of a face dual to vertex $i$ in the dual cell decomposition. For a closed discrete 1-form $\eta$, one can consider its periods
\[
\sum_{\gamma_1} \eta = \mathfrak{a}_1 \quad \sum_{\gamma_2} \eta = \mathfrak{a}_2 
\]
where the summations are respectively over a collection of oriented edges of the dual cell decomposition that form loops homologous to $\gamma_1$ and $\gamma_2$. The closeness condition implies that the summations are path-independent.  

On the other hand, a discrete 1-form is co-closed if for every primal face $\phi \in F'$, the summation
\[
\sum_{ij \in \partial \phi} \frac{1}{c_{ij}}\eta_{ij} =0.
\]
For a co-closed discrete 1-form $\eta$, one can also consider its periods
\[
\sum_{\gamma_1}  \frac{1}{c_{ij}}\eta = \mathfrak{b}_1 \quad \sum_{\gamma_2}  \frac{1}{c_{ij}} \eta = \mathfrak{b}_2 
\]
where the summations are respectively over a collection of oriented edges of the primal cell decomposition that form loops homologous to $\gamma_1$ and $\gamma_2$. 

A discrete 1-form is harmonic if it is closed and co-closed. It induces a linear mapping on the period space
\begin{align*}
	\mathfrak{h}_X : \mathbb{R}^{2} &\to \mathbb{R}^{2} \\
	(\mathfrak{a}_1,\mathfrak{a}_2) & \mapsto (\mathfrak{b}_1,\mathfrak{b}_2)
\end{align*}
which maps the periods of a harmonic 1-form $\eta$ over the dual cell decomposition to the periods of $\frac{1}{c}\eta$ over the primal cell decomposition. Because the edge weights are positive, one can use the maximum principle to show that $\mathfrak{h}_X$ is bijective (See \cite[Theorem 3.9]{BS2014}).

For a discrete harmonic 1-form $\eta$, its discrete Dirichlet energy $\mathcal{E}_c(\eta)$ can be expressed in terms of the periods
\[
\mathcal{E}_c(\eta):= \sum_{ij\in E} \frac{1}{c_{ij}} \eta^2_{ij} = \omega( \mathfrak{a}, \mathfrak{h}_{X}(\mathfrak{a})).
\]
Generally, if $\tilde{\eta}$ is another harmonic 1-form with period $\tilde{\mathfrak{a}}$ then we have
\begin{equation}\label{eq:diswitch}
	\omega( \mathfrak{a}, \mathfrak{h}_{X}(\tilde{\mathfrak{a}})) =  \sum_{ij\in E} \eta_{ij} \frac{\tilde{\eta}_{ij}}{c_{ij}} = \sum_{ij\in E}  \tilde{\eta}_{ij} \frac{\eta_{ij}}{c_{ij}} = \omega( \tilde{\mathfrak{a}}, \mathfrak{h}_{X}(\mathfrak{a})).
\end{equation}

\subsection{Infinitesimal deformations of circle patterns}

We recall results in \cite{Lam2021} relating discrete harmonic 1-forms to infinitesimal deformations of the circle pattern and explain how $T_X P(\Theta)$ is isomorphic to the period space of discrete harmonic 1-forms.

Recall we take $(V',E',F')$ the Delaunay cell decomposition by removing edges $ij$ where $\Theta_{ij}=0$. We write $(\tilde{V}',\tilde{E}',\tilde{F}')$ the lift of the cell decomposition to the universal cover. Let $X \in P(\Theta)$ represent circle patterns on a non-Euclidean affine torus with holonomy in the form of \eqref{eq:holmat}. We consider the Euclidean radii of the circumcircles for every face $R:\tilde{F}' \to \mathbb{R}_{>0}$, which satisfies
\[
R\circ \gamma_1 = e^{A_1} R, \quad R\circ \gamma_2 = e^{A_2} R
\] 
Suppose $\tilde{X} \in P(\Theta)$ is another circle pattern on a non-Euclidean affine torus with Euclidean radii $\tilde{R}:\tilde{F}' \to \mathbb{R}_{>0}$. Then we define $u:\tilde{F}' \to \mathbb{R}$ via
\[
u:= \log \frac{\tilde{R}}{R}.
\] 
It satisfies
\begin{equation}\label{eq:uperiod}
	u\circ \gamma_1 = u + \tilde{A}_1 - A_1, \quad u\circ \gamma_2 = u + \tilde{A}_2 - A_2
\end{equation}
For each circumcircle, there is a compatitbility condition on the radii of the adjacent circles in order to lay them out consistently.
Let $\phi_0$ be any face with circumradius $R_0$. We denote the circumradii of its neighboring faces as $R_1,R_2,R_3,\dots,R_n$. We also write $\Theta_1,\Theta_2,\dots,\Theta_n$ the intersection angles of these neighboring faces with the central face $\phi_0$. Since the angle sum at the circumcenter of the central face $\phi_0$ is $2\pi$, the function $u: F \to \mathbb{R}$ satisfies for every face $\phi_0 \in F$
\begin{align}\label{eq:nonu}
	\sum_{k=1}^{n} \cot^{-1}\left(\frac{1}{\sin \Theta_{k}}(\frac{e^{u_0} R_0}{e^{u_k} R_k}+ \cos  \Theta_{k}) \right) = \pi.
\end{align}
Conversely, any function $u:\tilde{F} \to\mathbb{R}$ satisfying \eqref{eq:nonu} over a circle pattern with radius $R$ determines locally a new circle pattern sharing the same intersection angle.

Suppose we have a 1-parameter family of circle patterns $X^{(t)}\in P(\Theta)$ with raddi $R^{(t)}$ satisfying $X=X^{(t=0)}$, $R=R^{(t=0)}$ and $\dot{R}:= \frac{d}{dt} R^{(t)}|_{t=0}$. We consider $\dot{u}:\tilde{F}'\to \mathbb{R}$ via
\[
\dot{u} = \frac{\dot{R}}{R}
\]
and define a discrete 1-form $\eta$ such that for every oriented edge $ij$
\[
\eta(ij)=\dot{u}_{ij,l} - \dot{u}_{ij,r}
\]
where $(ij,l)$ and $(ij,r)$ are the left and the right face of $ij$. By construction, $\eta$ is a closed 1-form. Equation \eqref{eq:uperiod} implies that $\eta$ is invariant under deck transformations and hence a well-defined closed 1-form on the torus with periods
\[
\sum_{\gamma_1} \eta = \dot{A}_1 \quad \sum_{\gamma_2} \eta = \dot{A}_2.
\]
Differentiating Equation \ref{eq:nonu} yields that $\eta$ is co-closed. Thus $\eta$ is a discrete harmonic 1-form. Conversely, by reversing the construction, every discrete harmonic 1-form corresponds to an infinitesimal deformation of the circle pattern, i.e. a tangent vector in $T_X P(\Theta)$. We summarize the results in \cite{Lam2021} relating the isomorphism between $T_XP(\Theta)$ and the period space $\mathbb{R}^2$.

\begin{proposition}\cite{Lam2021}\label{prop:gtpaper}
	Let $X \in P(\Theta)$ represent circle patterns on a non-Euclidean affine torus with holonomy in the form of \eqref{eq:holmat}. Then every $\mathfrak{a} \in \mathbb{R}^2$ corresponds to an infinitesimal deformation $x \in T_{X}P(\Theta)$ whose change in the affine holonomy is for $r=1,2$
	\[
		\dot{\rho}_r \rho^{-1}_r =  \left(\begin{array}{cc}
		\frac{\mathfrak{a}_r+ \mathbf{i} \mathfrak{b}_r}{2} & 0 \\  0 & 	-\frac{\mathfrak{a}_r+ \mathbf{i} \mathfrak{b}_r}{2}
	\end{array} \right)
	\]
	where $\mathfrak{b}= \mathfrak{h}_X (\mathfrak{a})$.
	
	For any nonzero $\mathfrak{a} \in \mathbb{R}^2$, we have
\begin{equation}\label{eq:ineq2021}
\omega(\mathfrak{a},\mathfrak{h}_X (\mathfrak{a}))= \sum_{ij} \frac{1}{c_{ij}} \eta^2_{ij} > \iint_S \eta^{\dagger} \wedge \star \eta^{\dagger} = \omega(\mathfrak{h}_X(\mathfrak{a}), \mathfrak{h}_{\tau}\mathfrak{h}_X(\mathfrak{a})) 
\end{equation}
where $\eta$ is a discrete harmonic 1-form with period $\mathfrak{a}=(\mathfrak{a}_1,\mathfrak{a}_2)$ and $\eta^{\dagger}$ is a smooth harmonic 1-form with period $\mathfrak{h}^{-1}_{\tau}\mathfrak{h}_X(\mathfrak{a})= - \mathfrak{h}_{\tau}\mathfrak{h}_X(\mathfrak{a})$. In other words, $\star \eta^{\dagger}$ has the same period as $\frac{\eta}{c}$.
\end{proposition}

In \cite{Lam2021}, the inequality (Equation \eqref{eq:ineq2021}) is deduced from two observations. First, the discrete Dirichlet energy of $\frac{\eta}{c}$ coincides with the classical energy of its piecewise-linear extension. Second, over closed 1-forms with prescribed periods, the unique minimizer of the classical energy is a smooth harmonic 1-form. 

\begin{corollary}\label{cor:nonid}For any non-zero $\mathfrak{a} \in \mathbb{R}^2$, we have
	\[
	|| \mathfrak{a} ||_{\tau} > || \mathfrak{h}_X (\mathfrak{a}) ||_{\tau} 
	\]
	In particular, $\mathfrak{h}_{X} \circ \mathfrak{h}_{X}  \neq - \mbox{Id}$.
\end{corollary}
\begin{proof}
	By the Cauchy–Schwarz inequality and Proposition \ref{prop:gtpaper},
	\begin{align*}
		|| \mathfrak{a} ||_{\tau}  \cdot || \mathfrak{h}_X (\mathfrak{a}) ||_{\tau} =	|| \mathfrak{a} ||_{\tau}  \cdot || -\mathfrak{h}_{\tau} \mathfrak{h}_X (\mathfrak{a}) ||_{\tau} \geq & \langle a, -\mathfrak{h}_{\tau} \mathfrak{h}_X(\mathfrak{a}) \rangle_{\tau} \\
		 = &\omega(a,  \mathfrak{h}_X(\mathfrak{a})) \\
		 > &\omega(\mathfrak{h}_X(\mathfrak{a}), \mathfrak{h}_{\tau}\mathfrak{h}_X(\mathfrak{a})) \\
		 =& || \mathfrak{h}_X (\mathfrak{a}) ||^2_{\tau}.
	\end{align*}
Since $\mathfrak{h}_X$ is injective and hence $ || \mathfrak{h}_X (\mathfrak{a}) ||_{\tau} \neq 0$ whenever $\mathfrak{a}$ nonzero, we obtain the claim.
\end{proof}

We are ready to prove the non-degeneracy of the bilinear form. 

\begin{proof}[Proof of Theorem \ref{thm:main2}]
	We take the parametrizion of $P(\Theta)$ as in Corollary \ref{cor:puncplane}. Fix $(A_1,A_2) \neq (0,0)$. Consider the circle pattern $X \in P(\Theta)$ on an non-Euclidean affine torus with holonomy in the form of \eqref{eq:holmat} with the given $(A_1,A_2)$. The tangent space $T_X P(\Theta)$ is isomorphic to $\mathbb{R}^2$ corresponding to the infinitesimal change of $(A_1,A_2)$. For any $x \in T_X P(\Theta)$, we write  $\mathfrak{a}=(\dot{A}_1,\dot{A}_2) \in \mathbb{R}^2$. Writing $\mathfrak{b}=\mathfrak{h}_X(\mathfrak{a})$ and $\tilde{\mathfrak{b}}=\mathfrak{h}_X(\tilde{\mathfrak{a}})$, Corollary \ref{cor:pullback} and Proposition \ref{prop:gtpaper} imply that
	\begin{align*}
	\frac{1}{2}\tilde{f}^*\omega_{P}(x,\tilde{x}) =&  (\mathfrak{a}_1+ \mathbf{i} \mathfrak{b}_1) (\tilde{\mathfrak{a}}_2+ \mathbf{i} \tilde{\mathfrak{b}}_2) - (\mathfrak{a}_2+ \mathbf{i} \mathfrak{b}_2) (\tilde{\mathfrak{a}}_1+ \mathbf{i} \tilde{\mathfrak{b}}_1) \\
	   =& \omega(\mathfrak{a}, \tilde{\mathfrak{a}}) - \omega(\mathfrak{b},\tilde{\mathfrak{b}}) + \mathbf{i} \omega(\mathfrak{b},\tilde{\mathfrak{a}}) + \mathbf{i} \omega(\mathfrak{a},\tilde{\mathfrak{b}}) \\
	   =& \omega(\mathfrak{a}, \tilde{\mathfrak{a}}) - \omega(\mathfrak{h}_X(\mathfrak{a}),\mathfrak{h}_X(\tilde{\mathfrak{a}})) + \mathbf{i} \omega(\mathfrak{h}_X(\mathfrak{a}),\tilde{\mathfrak{a}}) + \mathbf{i} \omega(\mathfrak{a},\mathfrak{h}_X(\tilde{\mathfrak{a}})) \\
	   =& \omega(\mathfrak{a}, \tilde{\mathfrak{a}}) - \omega(\mathfrak{h}_X(\mathfrak{a}),\mathfrak{h}_X(\tilde{\mathfrak{a}})) 
	\end{align*}
   where Equation \eqref{eq:diswitch} is used.
	
	We shall argue that $\tilde{f}^*\omega_{P}$ is non-degenerate on $T_X P(\Theta)$ and we prove it by contradiction. Assume $\tilde{f}^*\omega_P$ degenerates. Take any nonzero $\mathfrak{a}$. Observe that $\mathfrak{a}$ and $\tilde{\mathfrak{a}}:=\mathfrak{h}_X(\mathfrak{a})$ are linearly independent since
	\[
	\omega( \mathfrak{a}, \mathfrak{h}_X(\mathfrak{a})) >0
	\]
	is the discrete Dirichlet energy of a non-trivial discrete harmonic 1-form. The degeneracy of $\tilde{f}^*\omega_{P}$ implies that
	\begin{align*}
	0=& \omega(\mathfrak{a}, \tilde{\mathfrak{a}}) - \omega(\mathfrak{h}_X(\mathfrak{a}),\mathfrak{h}_X(\tilde{\mathfrak{a}})) \\
	=& \omega(\mathfrak{a} + \mathfrak{h}^2_X (\mathfrak{a}),\mathfrak{h}_X(\mathfrak{a})).
	\end{align*}
It yields that the vector $\mathfrak{a} + \mathfrak{h}^2_X (\mathfrak{a})$ is a multiple of $\mathfrak{h}_X(\mathfrak{a})$ for any $\mathfrak{a}$ since $\omega$ is non-degenerate on $\mathbb{R}^2$. Thus there exists a constant $\alpha \in \mathbb{R}$ such that
\[
\mbox{Id} + \mathfrak{h}_X^2 = \alpha \mathfrak{h}_X
\]
where $\mbox{Id}$ is the 2 by 2 identity matrix. Equation \eqref{eq:diswitch} implies that the constant $\alpha$ must be $0$ since for any $\mathfrak{a},\tilde{\mathfrak{a}} \in \mathbb{R}^2$,
\[
\omega( (\mbox{Id} + \mathfrak{h}_X^2)\mathfrak{a} , \tilde{\mathfrak{a}}) = \omega( \mathfrak{a} , (\mbox{Id} + \mathfrak{h}_X^2)\tilde{\mathfrak{a}})
\]
while
\[
\omega( \mathfrak{h}_X(\mathfrak{a}) , \tilde{\mathfrak{a}}) = -\omega( \mathfrak{a} , \mathfrak{h}_X(\tilde{\mathfrak{a}})).
\]
It yields that the bilinear form $\tilde{f}^*\omega_{P}$ being degenerate on $T_X P(\Theta)$ implies that $\mathfrak{h}_X \circ \mathfrak{h}_X=-\mbox{Id}$, which contradicts Corollary \ref{cor:nonid}. Hence $\tilde{f}^*\omega_{P}$ must be non-degenerate on $T_X P(\Theta)$. 
\end{proof}

\section{Example on Euclidean tori}

At the Euclidean structure, the bilinear form $\tilde{f}^*\omega_P$ cannot be reduced to $\omega$ and hence the argument in Section \ref{sec:harcon} fails. A conceptual reason is that the character variety of $P(S)$ degenerates at the set of Euclidean structures and hence investigating the change of holonomy at Euclidean structures have to be handled differently. Anyhow, we believe $\omega_P$ is still non-degenerate at Euclidean tori. 

We illustrate it by considering certain Delaunay angle structures where $P(\Theta)$ can be described explicitly. Suppose $(V,E,F)$ is a triangulated torus obtained via a quotient of the triangular lattice by translation. Combinatorially, we partition the edge set $E$ into three subsets $E_1,E_2,E_3$, where two edges belong to the same $E_i$ if they are ``parallel" (See Figure \ref{fig:tritorus}). Fix any three angles $\Theta_1,\Theta_2,\Theta_3 \in [0,\pi)$ such that $ \Theta_1+ \Theta_2 + \Theta_3 = \pi$. We define a Delaunay angle structure $\Theta:E \to [0,\pi)$ via
\begin{align*}
\Theta_{ij}:= \begin{cases}
	\Theta_1 \quad \text{if} \quad  ij \in E_1\\ \Theta_2  \quad \text{if} \quad  ij \in E_2 \\
	\Theta_3  \quad \text{if} \quad  ij \in E_3
\end{cases}
\end{align*}
Then the space of circle patterns $P(\Theta)\cong \mathbb{R}^2$ is the collection of all $X:E \to \mathbb{C}$ in the form
\begin{align*}
	X_{ij}:= \begin{cases}
		\alpha e^{\mathbf{i} \Theta_1} \quad \text{if} \quad  ij \in  E_1\\ \beta e^{\mathbf{i} \Theta_2}  \quad \text{if} \quad  ij \in E_2 \\
		\frac{1}{\alpha \beta} e^{\mathbf{i} \Theta_3}  \quad \text{if} \quad  ij \in E_3
	\end{cases}
\end{align*}
where $\alpha, \beta$ are positively real numbers. In this way, one can write down the logarithmic derivatives explicitly for the tangent vectors in $T_X P(\Theta)$ and apply Theorem \ref{thm:penner} together with Corollary \ref{cor:penner} to verify that $\tilde{f}^*\omega_{P}$ is non-degenerate everywhere on $P(\Theta)$, particularly including the unique circle pattern $X^{\dagger}$ on an Euclidean torus.

\begin{figure}
	\begin{tikzpicture}[line cap=round,line join=round,>=triangle 45,x=1.5cm,y=1.5cm]
	\clip(-2.0479707411772948,0.919540160353863734) rectangle (3.04788568399187954,3.046084852190591);
	\draw [line width=.5pt] (-2.,1.)-- (-1.,1.);
	\draw [line width=.5pt] (-1.5,1.0293044143525825) -- (-1.5,0.9706955856474176);
	\draw [line width=.5pt] (-1.,1.)-- (0.,1.);
	\draw [line width=.5pt] (-0.5,1.0293044143525825) -- (-0.5,0.9706955856474176);
	\draw [line width=.5pt] (0.,1.)-- (1.,1.);
	\draw [line width=.5pt] (0.5,1.0293044143525825) -- (0.5,0.9706955856474176);
	\draw [line width=.5pt] (-1.,2.)-- (0.,2.);
	\draw [line width=.5pt] (-0.5,2.0293044143525827) -- (-0.5,1.9706955856474178);
	\draw [line width=.5pt] (0.,2.)-- (1.,2.);
	\draw [line width=.5pt] (0.5,2.0293044143525827) -- (0.5,1.9706955856474178);
	\draw [line width=.5pt] (1.,2.)-- (2.,2.);
	\draw [line width=.5pt] (1.5,2.0293044143525827) -- (1.5,1.9706955856474178);
	\draw [line width=.5pt] (0.,3.)-- (1.,3.);
	\draw [line width=.5pt] (0.5,3.0293044143525827) -- (0.5,2.9706955856474178);
	\draw [line width=.5pt] (1.,3.)-- (2.,3.);
	\draw [line width=.5pt] (1.5,3.0293044143525827) -- (1.5,2.9706955856474178);
	\draw [line width=.5pt] (2.,3.)-- (3.,3.);
	\draw [line width=.5pt] (2.5,3.0293044143525827) -- (2.5,2.9706955856474178);
	\draw [line width=.5pt] (-2.,1.)-- (-1.,2.);
	\draw [line width=.5pt] (-1.5293552459854993,1.5120874542293232) -- (-1.4879125457706768,1.4706447540145007);
	\draw [line width=.5pt] (-1.512087454229323,1.5293552459854995) -- (-1.4706447540145005,1.487912545770677);
	\draw [line width=.5pt] (-1.,2.)-- (0.,3.);
	\draw [line width=.5pt] (-0.5293552459854994,2.5120874542293232) -- (-0.48791254577067694,2.4706447540145007);
	\draw [line width=.5pt] (-0.5120874542293232,2.5293552459854998) -- (-0.4706447540145007,2.4879125457706768);
	\draw [line width=.5pt] (1.,3.)-- (0.,2.);
	\draw [line width=.5pt] (0.5293552459854994,2.4879125457706768) -- (0.4879125457706764,2.5293552459854998);
	\draw [line width=.5pt] (0.5120874542293238,2.4706447540145007) -- (0.4706447540145007,2.5120874542293232);
	\draw [line width=.5pt] (-1.,1.)-- (0.,2.);
	\draw [line width=.5pt] (-0.5293552459854994,1.5120874542293232) -- (-0.48791254577067694,1.4706447540145007);
	\draw [line width=.5pt] (-0.5120874542293232,1.5293552459854995) -- (-0.4706447540145007,1.487912545770677);
	\draw [line width=.5pt] (1.,2.)-- (2.,3.);
	\draw [line width=.5pt] (1.4706447540145002,2.5120874542293232) -- (1.5120874542293234,2.4706447540145007);
	\draw [line width=.5pt] (1.487912545770676,2.5293552459854998) -- (1.529355245985499,2.4879125457706768);
	\draw [line width=.5pt] (1.,2.)-- (0.,1.);
	\draw [line width=.5pt] (0.5293552459854994,1.487912545770677) -- (0.4879125457706764,1.5293552459854995);
	\draw [line width=.5pt] (0.5120874542293238,1.4706447540145007) -- (0.4706447540145007,1.5120874542293232);
	\draw [line width=.5pt] (1.,1.)-- (2.,2.);
	\draw [line width=.5pt] (1.4706447540145002,1.5120874542293232) -- (1.5120874542293234,1.4706447540145007);
	\draw [line width=.5pt] (1.487912545770676,1.5293552459854995) -- (1.529355245985499,1.487912545770677);
	\draw [line width=.5pt] (2.,2.)-- (3.,3.);
	\draw [line width=.5pt] (2.470644754014501,2.5120874542293232) -- (2.5120874542293246,2.4706447540145007);
	\draw [line width=.5pt] (2.4879125457706768,2.5293552459854998) -- (2.529355245985499,2.4879125457706768);
	\draw [line width=.5pt] (-1.,2.)-- (1.,3.);
	\draw [line width=.5pt] (-0.034947553351037296,2.515289554591079) -- (-0.008736888337759324,2.4628682245645233);
	\draw [line width=.5pt] (-0.013105332506638987,2.526210665013278) -- (0.013105332506638987,2.4737893349867224);
	\draw [line width=.5pt] (0.008736888337759324,2.537131775435477) -- (0.03494755335103702,2.4847104454089215);
	\draw [line width=.5pt] (0.,2.)-- (2.,3.);
	\draw [line width=.5pt] (0.9650524466489632,2.515289554591079) -- (0.9912631116622406,2.4628682245645233);
	\draw [line width=.5pt] (0.9868946674933612,2.526210665013278) -- (1.0131053325066386,2.4737893349867224);
	\draw [line width=.5pt] (1.0087368883377592,2.537131775435477) -- (1.0349475533510366,2.4847104454089215);
	\draw [line width=.5pt] (1.,2.)-- (3.,3.);
	\draw [line width=.5pt] (1.9650524466489634,2.515289554591079) -- (1.9912631116622408,2.4628682245645233);
	\draw [line width=.5pt] (1.9868946674933614,2.526210665013278) -- (2.0131053325066386,2.4737893349867224);
	\draw [line width=.5pt] (2.008736888337759,2.537131775435477) -- (2.034947553351037,2.4847104454089215);
	\draw [line width=.5pt] (0.,1.)-- (2.,2.);
	\draw [line width=.5pt] (0.9650524466489632,1.5152895545910792) -- (0.9912631116622406,1.462868224564523);
	\draw [line width=.5pt] (0.9868946674933612,1.526210665013278) -- (1.0131053325066386,1.4737893349867222);
	\draw [line width=.5pt] (1.0087368883377592,1.5371317754354772) -- (1.0349475533510366,1.484710445408921);
	\draw [line width=.5pt] (-1.,1.)-- (1.,2.);
	\draw [line width=.5pt] (-0.034947553351037296,1.5152895545910792) -- (-0.008736888337759324,1.462868224564523);
	\draw [line width=.5pt] (-0.013105332506638987,1.526210665013278) -- (0.013105332506638987,1.4737893349867222);
	\draw [line width=.5pt] (0.008736888337759324,1.5371317754354772) -- (0.03494755335103702,1.484710445408921);
	\draw [line width=.5pt] (-2.,1.)-- (0.,2.);
	\draw [line width=.5pt] (-1.0349475533510373,1.5152895545910792) -- (-1.0087368883377592,1.462868224564523);
	\draw [line width=.5pt] (-1.0131053325066388,1.526210665013278) -- (-0.9868946674933609,1.4737893349867222);
	\draw [line width=.5pt] (-0.9912631116622406,1.5371317754354772) -- (-0.9650524466489626,1.484710445408921);
	\begin{scriptsize}
		\draw [fill=black] (-2.,1.) circle (.5pt);
		\draw [fill=black] (-1.,1.) circle (.pt);
		\draw [fill=black] (0.,1.) circle (.5pt);
		\draw [fill=black] (1.,1.) circle (.5pt);
		\draw [fill=black] (-1.,2.) circle (.5pt);
		\draw [fill=black] (0.,3.) circle (.5pt);
		\draw [fill=black] (0.,2.) circle (.5pt);
		\draw [fill=black] (1.,2.) circle (.5pt);
		\draw [fill=black] (2.,2.) circle (.5pt);
		\draw [fill=black] (1.,3.) circle (.5pt);
		\draw [fill=black] (2.,3.) circle (.5pt);
		\draw [fill=black] (3.,3.) circle (.5pt);
	\end{scriptsize}
\end{tikzpicture}
\caption{A triangulated torus obtained via a quotient of the triangular lattice by translation combinatorially. Its edges are partitioned into three types. Edges of the same type are drawn in the same style. }
\label{fig:tritorus}
\end{figure}
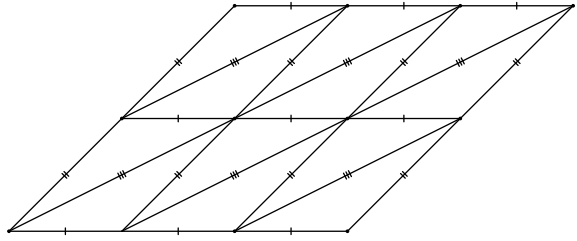

\bibliographystyle{plainurl}
\bibliography{symplectic}

\end{document}